\documentclass[review]{elsarticle}

\usepackage{lineno,hyperref}
\usepackage{amstext}
\usepackage{amsthm}
\usepackage{amssymb}
\modulolinenumbers[5]

\begin{document}

\begin{frontmatter}

\title{IP$^{*}$-sets in function field and mixing properties}

\author[rvt]{Dibyendu De\corref{cor1}\fnref{fn1}}
\ead{dibyendude@klyuniv.ac.in}


\fntext[fn1]{This author is partially supported by DST-PURSE programme.}

\address[rvt]{Department of Mathematics, University of Kalyani, Kalyani-741235,
West Bengal, India}

\author[focal]{Pintu Debnath\fnref{fn2}}
\ead{pintumath1989@gmail.com}
\address[focal]{Department of Mathematics, Bashirhat College, West Bengal, India}
\fntext[fn2]{The work of this article was a part of this author's Ph.D. dissertation which was supported by a CSIR Research Fellowship.}
%\author[els]{S.~Pepping\corref{cor2}\fnref{fn1,fn3}}
%\ead[url]{http://www.elsevier.com}
%\fntext[fn3]{Yet another author footnote. Indeed, you can have
%any number of author footnotes.}
%\address[els]{Central Application Management,
%Elsevier, Radarweg 29, 1043 NX\\
%Amsterdam, Netherlands}

\newtheorem{thm}{Theorem}
\newtheorem{lem}[thm]{Lemma}
\newtheorem{cor}[thm]{Corollary}
\newtheorem{defn}[thm]{Definition}
\newdefinition{exam}{Example}
\newtheorem{qs}{Question}
\newdefinition{rmk}{Remark}
\newproof{pf}{Proof}
\newproof{pot}{Proof of Theorem \ref{thm2}}

\newcommand{\R}{{\mathbb R}}
\newcommand{\Q}{{\mathbb Q}}

\newcommand{\T}{{\mathbb T}}
\newcommand{\C}{{\mathbb C}}
\newcommand{\tto}{{\longrightarrow}}
\newcommand{\emp}{\emptyset}

\begin{abstract}
The ring of polynomial over a finite field $F_q[x]$ has received much attention, both from a combinatorial  viewpoint   as in regards to its action on  measurable dynamical systems. In the case of $(\mathbb{Z},+)$ we know that the ideal generated by any nonzero element is an IP$^*$-set.  In the present article  we first establish that the analogous result is true for $F_q[x]$.  We further use this result to establish some mixing properties of the action of $(F_q[x],+)$. We shall also discuss on Khintchine's recurrence for the action of $(F_q[x]\setminus\{0\},\cdot)$.
\end{abstract}


\end{frontmatter}

%\linenumbers

\section{Introduction}

By a measurable dynamical system (MDS), we mean $\left(X,\mathcal{B},\mu,(T_{g})_{g\in G}\right)$,
where $\left(X,\mathcal{B},\mu\right)$ is a probability space and
for each $g\in G$, $T_{g}:X\to X$ is an invertible and measure preserving transformation.
For an MDS $\left(X,\mathcal{B},\mu,(T_{g})_{g\in G}\right)$, let
${\cal B}^{+}$ be the set of all positive measured sets, and $N(A,B)=\{g\in G:\mu(A\cap T_{g}B)\neq0\}$.
The classical results in ergodic theory state that a transformation
$(T_{g})_{g\in G}$ is ergodic iff $N(A,B)\neq\emptyset$ for each
pair of $A,B\in{\cal B}^{+}$, weakly mixing iff $\{g\in G:|\mu(A\cap T_{g}B)-\mu(A)\mu(B)|<\epsilon\}$
is a central$^{*}$-set for each pair of $A,B\in{\cal B}^{+}$, mildly
mixing iff $\{g\in G:|\mu(A\cap T_{g}B)-\mu(A)\mu(B)|<\epsilon\}$ is
an IP$^{*}$-set, and strongly mixing iff $\{g\in G:|\mu(A\cap T_{g}B)-\mu(A)\mu(B)|<\epsilon\}$
is a cofinite set.  See for example \citep{refBD,refKY08}.
In \citep{refKY08} authors described the notions in terms of families.
We are here interested with the family of central$^{*}$-sets, IP$^{*}$-sets,
 cofinite sets and difference sets which we shall denote by ${\cal C}^{*}$, ${\cal IP}^{*}$,
 ${\cal C}_f$ and $\triangle$ respectively. The notions of ${\cal C}^{*}$, ${\cal IP}^{*}$ will be defined latter.  In these terms ${\cal C}^{*}$-mixing
implies weak mixing, ${\cal IP}^{*}$-mixing implies mild mixing and
${\cal C}_{f}$-mixing implies strong mixing.
\begin{defn} Let $A$ be a subset of a semigroup $S$.

(1) $A$ is called an IP-set if there is a subsequence $\langle x_n\rangle_{n=1}^{\infty}$ such that all finite subset $F\in\mathcal{P}_f(\mathbb{N})$
sums of forms $\sum_{n\in F}x_n$  are in $A$. A subset $A\subset S$ is said to be an IP$^*$-set, if it meets every IP-set in $S$. The collection
of all $\mbox{IP}$-sets is denoted by $\mathcal{IP}$ and the collection of all IP$^*$-sets will be denoted by  $\mathcal{IP}^*$.\\

(2) $D$ is called a $\triangle$-set if it contains an infinite difference set, i.e. there is a
sequence $\langle x_n\rangle_{n=1}^{\infty}$ of $S$ such that $D \supset \triangle\left(\langle x_n\rangle_{n=1}^{\infty}\right) =\{ x_n\cdot x_m ^{-1}:m,n\in\mathbb{N}\}$.  A subset $D\subset S$ is said to be an $\triangle^*$-set, if it meets every $\triangle$-set in $S$. The collection of $\triangle$-sets is denoted by $\bigtriangleup$ and the collection of all $\triangle^*$-sets will be denoted by  $\bigtriangleup^*$.
\end{defn}

In recent years $(F_{q}[x],+)$, where $F_{q}[x]$ denotes the ring
of all polynomials over the finite field of characteristic $q$, has
received much attention both from a combinatorial  viewpoint   as in regards to its action on  measurable dynamical systems.
In \citep{refL} the author has proved a version of celebrated Green-Tao
Theorem for $(F_{q}[x],+)$. In \citep{refBTZ} the authors  proved
some higher order versions of Khintchine's recurrence theorem for the action
of $(F_{q}[x],+)$. In the present article we shall present some combinatorial
properties of $(F_{q}[x],+)$ and $(F_{q}[x],\cdot)$. Further we
will apply these combinatorial properties to find some interesting
properties of their actions on measure space.

In order to discuss combinatorial properties of $(F_{q}[x],+)$ and $(F_{q}[x],\cdot)$
we shall need  algebraic properties of its Stone-\v{C}ech compactification of $F_{q}[x]$.
For this purpose we need to discuss the algebra of the Stone \v{C}ech
compactification of a discrete semigroup $S$. For a  discrete semigroup $S$, the Stone-\v{C}ech compactification of $S$ will be denoted by $\beta S$.    We take the points
of $\beta S$ to be the ultrafilters on $S$, identifying the principal
ultrafilters with the points of $S$ and thus pretending that $S\subseteq\beta S$.
Given $A\subseteq S$, we denote
\[
\overline{A}=\{p\in\beta S:A\in p\}.
\]

The set $\{\overline{A}:A\subset S\}$ is a basis for the closed sets
of $\beta S$. The operation $`\cdot$' on $S$ can be extended to the
Stone-\v{C}ech compactification $\beta S$ of $S$ so that $(\beta S,\cdot)$
is a compact right topological semigroup (meaning that for any $p\in\beta S$,
the function $\rho_{p}:\beta S\rightarrow\beta S$ defined by $\rho_{p}(q)=q\cdot p$
is continuous) with $S$ contained in its topological center (meaning
that for any $x\in S$, the function $\lambda_{x}:\beta S\rightarrow\beta S$
defined by $\lambda_{x}(q)=x\cdot q$ is continuous). A nonempty subset
$I$ of a semigroup $T$ is called a \textit{left ideal} of $S$ if
$TI\subset I$, a \textit{right ideal} if $IT\subset I$, and a \textit{two
sided ideal} (or simply an \textit{ideal}) if it is both a left and
right ideal. A \textit{minimal left idea}l is the left ideal that
does not contain any proper left ideal. Similarly, we can define \textit{minimal
right ideal} and \textit{smallest ideal}.

Any compact Hausdorff right topological semigroup $T$ has the smallest
two sided ideal

\[
\begin{array}{ccc}
K(T) & = & \bigcup\{L:L\text{ is a minimal left ideal of }T\}\\
 & = & \,\,\,\,\,\bigcup\{R:R\text{ is a minimal right ideal of }T\}.
\end{array}
\]

Given a minimal left ideal $L$ and a minimal right ideal $R$, $L\cap R$
is a group, and in particular contains an idempotent. If $p$ and
$q$ are idempotents in $T$ we write $p\leq q$ if and only if $pq=qp=p$.
An idempotent is minimal with respect to this relation if and only
if it is a member of the smallest ideal $K(T)$ of $T$. Given $p,q\in\beta S$
and $A\subseteq S$, $A\in p\cdot q$ if and only if the set $\{x\in S:x^{-1}A\in q\}\in p$,
where $x^{-1}A=\{y\in S:x\cdot y\in A\}$. See \citep{refHS} for
an elementary introduction to the algebra of $\beta S$ and for any
unfamiliar details.

\begin{defn}Let $C$ be a subset of a semigroup $S$. Then $C$ is called a central set if there exists a minimal idempotent $p\in K(\beta S)$ such that $C\in p$. A subset of $S$, which meets every central set, called central$^*$ set.  We shall denote the class of all central sets as $\mathcal{C}$  and that of all central$^*$ sets as $\mathcal{C}^*$.
\end{defn}

The notion of central set was first introduced by Furstenberg in \citep{refF}
using topological dynamics and proved to be equivalent with definition
in \citep{refBH}. The basic fact that we need about central sets
is given by the Central Sets Theorem, which is due to Furstenberg
\citep[Proposition 8.21]{refF} for the case $S=\mathbb{Z}$.

\begin{thm}[Central Sets Theorem]
 \label{cst} Let $S$ be a semigroup. Let ${\cal T}$ be the set
of sequences $\langle y_{n}\rangle_{n=1}^{\infty}$ in $S$. Let $C$
be a subset of $S$ which is central and let $F\in\mathcal{P}_{f}(\mathcal{T})$.
Then there exist a sequence $\langle a_{n}\rangle_{n=1}^{\infty}$
in $S$ and a sequence $\langle H_{n}\rangle_{n=1}^{\infty}$ in $\mathcal{P}_{f}(\mathbb{N})$
such that for each $n\in\mathbb{N}$, $\max H_{n}<\min H_{n+1}$ and for each
$L\in\mathcal{P}_{f}(\mathbb{N})$ and each $f\in F$, $\sum_{n\in L}(a_{n}+\sum_{t\in H_{n}}f(t))\in C$.
\end{thm}

However, the most general version of Central Sets Theorem is presented in \citep{refDHS}, where all the sequences have been handled simultaneously.\\

To end this preliminary discussions let us recall Khintchine's Theorem,
which states that for any measure preserving system $\left(X,\mathcal{B},\mu,T\right)$,
and for any $\epsilon>0$ the set $\{n\in\mathbb{Z}:\mu(A\cap T^{-n}A)>\mu(A)^{2}-\epsilon\}$
is an IP$^{*}$-set and in particular syndetic. In \citep{refBHK} the authors proved
that for any ergodic system $\left(X,\mathcal{B},\mu,T\right)$ the
sets $\{n\in\mathbb{Z}:\mu(A\cap T^{-n}A\cap T^{-2n}A)>\mu(A)^{3}-\epsilon\}$
and $\{n\in\mathbb{Z}:\mu(A\cap T^{-n}A\cap T^{-2n}A\cap T^{-3n}A>\mu(A)^{4}-\epsilon\}$
are syndetic subsets of $\mathbb{Z}$. On the other hand they  proved
that for $n\geq4$ the above result does not hold in general.

In \citep{refBTZ} the authors proved result analogous to Khintchine's
Theorem. In fact they  proved that for $q>2$ if $c_{0},c_{1,}c_{2}$
are distinct elements of $ F_{q}[x]$ and $\left(X,{\cal B},\mu,T_{f\in F_{q}[x]}\right)$
is an ergodic system, then for any $A\in{\cal B}^+$, and $\epsilon>0$ the set
\[
\{f\in F_{q}[x]:\mu(T_{c_{0}f}A\cap T_{c_{1}f}A\cap T_{c_{2}f}A)>\mu(A)^{3}-\epsilon\}
\]

is syndetic.

The authors also proved that for $q>3$ and $c_{0},c_{1,}c_{2},c_{3}\in F_{q}[x]$
the above conclusion is true provided that $c_{i}+c_{j}=c_{k}+c_{l}$
for some permutation $\{i,j,k,l\}$ of $\{1,2,3,4\}$. Further analogous
to \citep{refBHK} the authors proved that for any $k\geq3$ there
exists $(c_{0},c_{1},\ldots,c_{k})\in F_{q}[x]^{k+1}$ for which Khintchine's
Theorem does not hold in general.\\

At the end of this article we shall present some observation on Khintchine's
Theorem for the action of $(F_{q}[x],\cdot)$.

\noindent \textbf{Acknowledgement} : We would like to thank Professor Neil Hindman
for his helpful suggestions. We also thank the referee for her/his  comments that have resulted in substantial improvements to this paper.

\section{Combinatorial Properties of $F_{q}[x]$}

In the terminology of Furstenberg \citep{refF}, an IP$^{*}$ set
$A$ in $\mathbb{Z}$ is a  set, which meets $\mbox{FS}\langle x_{n}\rangle_{n=1}^{\infty}$
for any sequence $\langle x_{n}\rangle_{n=1}^{\infty}$ in $\mathbb{Z}$.
 This in turn implies that $A$ is an IP$^{*}$ set iff it belongs to every
idempotent of $\beta\mathbb{Z}$. IP$^{*}$ -sets are known to have rich combinatorial structures. For example  IP$^{*}$- sets are always
syndetic. Given any IP$^{*}$-set $A$ which is a subset of the set of integers
$\mathbb{Z}$ and a sequence $\langle x_{n}\rangle_{n=1}^{\infty}$ in $\mathbb{Z}$
there exists a  sum subsystem $\langle y_n\rangle_{n=1}^{\infty}$
of $\langle x_{n}\rangle_{n=1}^{\infty}$ such that

\[
FS(\langle y_{n}\rangle_{n=1}^{\infty})\cup FP(\langle y_{n}\rangle_{n=1}^{\infty})\subseteq A,
\]
where  for any sequence $\langle x_{n}\rangle_{n=1}^{\infty}$ in $\mathbb{Z}$, FS$(\langle x_{n}\rangle_{n=1}^{\infty})$ is defined to be the set $\{\sum_{n\in F}x_n:F$ is a finite subset of $\mathbb{N}
\}$. FP$(\langle x_{n}\rangle_{n=1}^{\infty})$ can be defined analogously.

It is well known that in the ring $(\mathbb{Z},+,\cdot)$ the non
trivial principal ideals are IP$^{*}$ sets. So the natural question
is, whether this result is true for  arbitrary rings. The answer is
no. In fact in the ring $\mathbb{Z}[x]$ the ideal generated by $x$ has empty
intersection with $\mathbb{N}$, whereas $\mathbb{N}$ is an IP-set in $\mathbb{Z}[x]$. We will
prove that in the ring $(F_{q}[X],+,\cdot)$ every principal ideal is an  IP$^{*}$-set.
In fact we will also prove that in $(F_{q}\left[X_{1},X_{2},\ldots,X_{k}\right],+,\cdot)$
every ideal of the form $\langle f_{1}(X_{1}),f_{2}(X_{2}),\ldots,f_{k}(X_{k})\rangle$
is an IP$^{*}$-set.

\begin{thm}
\label{IPstar} In the polynomial ring $(F_{q}\left[X_{1},X_{2},\ldots,X_{k}\right],+,\cdot)$
over the finite field\textup{ $F_{q}$}, the ideal $\langle f_{1}(X_{1}),f_{2}(X_{2}),\ldots,f_{k}(X_{k})\rangle$
generated by $f_{1}(X_{1})$, \textup{$f_{2}(X_{2})$, $\ldots$,
$f_{k}(X_{k})$} (at most one of which is non constant), is an IP$^{*}$-set
in the corresponding additive group.
\end{thm}

\begin{proof}
For simplicity we work with $k=2$. Let $\langle g_{n}(X_{1},X_{2})\rangle_{n=1}^{\infty}$ be
a sequence in $F_{q}[X_{1},X_{2}]$.

Let $g(X_{1},X_{2})$ be a polynomial in $F_{q}[X_{1},X_{2}]$. Then

\[
g(X_{1},X_{2})=\sum_{i\leq n,j\leq m}a_{i,j}X_{1}^{i}X_{2}^{j},\mbox{ where }a_{i,j}\in F_{q}.
\]

Since $X_{1}^{i},f_{1}(X_{1})\in F_{q}[X_{1}]$ and $X_{2}^{j},f_{2}(X_{2})\in F_{q}[X_{2}]$
by applying a division algorithm we have
\[
X_{1}^{i}=f_{1}(X_{1})q_{1,i}(X_{1})+r_{1,i}(X_{1})\mbox{, where deg}(r_{1,i}(X_{1}))<\mbox{deg}f_{1}(X_{1})
\]
\[
X_{2}^{j}=f_{2}(X_{2})q_{2,j}(X_{2})+r_{2,j}(X_{2})\mbox{, where deg}(r_{2,j}(X_{2}))<\mbox{deg}f_{2}(X_{2}).
\]
Then $g(X_{1},X_{2})$ can be expressed as
\[
\sum_{i\leq n,j\leq m}a_{i,j}(f_{1}(X_{1})q_{1,i}(X_{1})+r_{1,i}(X_{1}))(f_{2}(X_{2})q_{2,j}(X_{2})+r_{2,j}(X_{2})).
\]
\[
\!\!\!\!\!\!\!\!\!\!\!\!\!\!\!\!\!\!\!\!\!\!\!\!\!\!\!\!\!\!g(X_{1},X_{2})=f_{1}(X_{1})h_{1}(X_{1},X_{2})+f_{2}(X_{2})h_{2}(X_{1},X_{2})
\]
\[
\,\,\,\,\,\,\,\,\,\,\,\,\,\,\,+\!\!\!\!\!\!\!\!\!\!\!\!\!\!\!\!\!\!\!\!\!\!\!\sum_{\begin{array}{c}
\mbox{ deg}(r_{1,i}(X_{1}))<\mbox{deg}f_{1}(X_{1})\\
\mbox{deg}(r_{2,j}(X_{2}))<\mbox{deg}f_{2}(X_{2})
\end{array}}\!\!\!\!\!\!\!\!\!\!\!\!\!\!\!\!\!\!\!\!\!\!\! a_{i,j}r_{1,i}(X_{1})r_{2,j}(X_{2}).
\]
Therefore we can write

\[
g(X_{1},X_{2})=h(X_{1},X_{2})+r(X_{1},X_{2}),
\]
where

\[
h(X_{1},X_{2})\in\langle f_{1}(X_{1}),f_{2}(X_{2})\rangle
\]
and $r(X_{1},X_{2})$ is a polynomial such that $\mbox{deg}\,r(X_{1},X_{2})<\mbox{deg}f_{1}(X_{1})+\mbox{deg}f_{2}(X_{2})$.

This implies that

\[
g_{n}(X_{1},X_{2})=h_{n}(X_{1},X_{2})+r_{n}(X_{1},X_{2})
\]
where

\[
h_{n}(X_{1},X_{2})\in\langle f_{1}(X_{1}),f_{2}(X_{2})\rangle
\]
and $r_{n}(X_{1},X_{2})$ is a polynomial such that $\mbox{deg}\,r_{n}(X_{1},X_{2})<\mbox{deg}f_{1}(X_{1})+\mbox{deg}f_{2}(X_{2})$.

But the set $\{r_{n}(X_{1},X_{2}):n\in\mathbb{N}\}$ is finite. Since $\{g_{n}(X_{1},X_{2}):n\in\mathbb{N}\}$
is infinite there exists $q$ many polynomials $g_{n_{i}}(X_{1},X_{2}):i=1,2,\ldots,q$
such that the corresponding $r_{n_{i}}(X_{1},X_{2})$ for $i=1,2,\ldots,q$
are equals. Now adding we get
\[
\sum_{i=1}^{q}g_{n_{i}}(X_{1},X_{2})=\sum_{i=1}^{q}h_{n_{i}}(X_{1},X_{2})+\sum_{i=1}^{q}r_{n_{i}}(X_{1},X_{2}).
\]
This implies that ${\displaystyle \sum_{i=1}^{q}g_{n_{i}}(X_{1},X_{2})}\in\langle f_{1}(X_{1}),f_{2}(X_{2})\rangle$
as ${\displaystyle \sum_{i=1}^{q}r_{n_{i}}(X_{1},X_{2})=0}$. Therefore,
$\langle f_{1}(X_{1}),f_{2}(X_{2})\rangle$ is an IP$^{*}$-set.
\end{proof}

 In case of $\mathbb{Z}$ we know that iterated spectra of an IP$^{*}$ set
are also IP$^{*}$ but may not contain any ideal \citep{refBHK96}. But
for $(F_{q}[X],+)$ any IP$^{*}$ set contains an ideal up to finitely many terms.

\begin{thm}
\label{syndetic IP}Any $IP^{*}$-set in $(F_{q}[X],+)$ contains
an ideal of the form $\langle X^{m}\rangle$, for some $m\in\mathbb{N}$,
up to finitely many terms.
\end{thm}

\begin{proof}
Let us claim that any syndetic IP set $A$ in $(F_{q}[X],+)$ contains
$\langle X^{m}\rangle$ for some $m\in\mathbb{N}$. Now $A$ being a syndetic
set will be of the form

\[
A={\displaystyle \bigcup_{i=1}^{k}(f_{i}(X)+\langle X^{m}\rangle)}\mbox{ (up to finitely many terms) }
\]
for some $m,k\in\mathbb{N}$ with $m>\mbox{deg}f_{i}(X)$. Again, since $A$
is an IP-set, one of $f_{i}(X)$ must be zero. In fact $A$ being an
IP set, there exist a sequence $\langle g_{i}(x)\rangle_{i=1}^{\infty}$ such that
$FS\langle g_{i}(x)\rangle\subseteq A$. This implies that for each
$i\in\mathbb{N}$, there exists $j\in\{1,2,\cdots,k\}$ and some $h_{i}(x)\in F_{q}[X]$
such that $g_{i}(X)=f_{j}(X)+h_{i}(X)X^{m}$. Since $\{g_{i}(X):i\in\mathbb{N}\}$
is infinite, there exist $q$ many polynomials $g_{n_{i}}(X):i=1,2,\ldots,q$
such that the corresponding $f_{j_{i}}(X)$ are equal for $i=1,2,\ldots,q$,
and such sum of $q$ many polynomials is equal to zero. Hence some
$f_{j}(X)$ is equal to zero.
\end{proof}

We end this section with the following observation. We know that in the
case of $\mathbb{Z}$, the intersection of thick set and an IP syndetic set may
not be central, but in case of $F_{q}[X]$, such sets will be always
central set. In fact in $F_{q}[X]$, any IP syndetic set is an  IP$^{*}$ set.  Again
since the thick sets are always central set, intersection of a thick set and
 an IP syndetic set is central set.

\section{Mixing Properties of the action of $(F_{q}[x],+)$}

In this section we shall show that all the mixing properties are equivalent
under the action of $(F_{q}[x],+)$. First let us recall the following
definitions.

\begin{defn}
A measure preserving system $\left(X,\mathcal{B},\mu,(T_{g})_{g\in G}\right)$
is said to be ergodic if for any set $A\in{\cal B}$ which satisfies $\mu(A\bigtriangleup T_{g}A)=0$
for any $g\in G$ has either measure $0$ or $1$.
\end{defn}

\begin{defn}
Let $\left(X,\mathcal{B},\mu,(T_{g})_{g\in G}\right)$ be a measure
preserving dynamical system. Then

\begin{enumerate}
\item $\left(X,\mathcal{B},\mu,(T_{g})_{g\in G}\right)$ is called strong
mixing if for any $\epsilon>0$ and any $A,B\in{\cal B}$ with positive
measure, the set $\{g\in G:|\mu(A\cap T_{g}B)-\mu(A)\mu(B)|<\epsilon\}$
 is a cofinite set.
\item $\left(X,\mathcal{B},\mu,(T_{g})_{g\in G}\right)$ is called mild
mixing if for any $\epsilon>0$ and any $A,B\in{\cal B}$ with positive
measure, the set $\{g\in G:|\mu(A\cap T_{g}B)-\mu(A)\mu(B)|<\epsilon\}$
is an IP$^{*}$-set.
\item $\left(X,\mathcal{B},\mu,(T_{g})_{g\in G}\right)$ is called weak
mixing if for any $\epsilon>0$ and any $A,B\in{\cal B}$ with positive
measure, the set $\{g\in G:|\mu(A\cap T_{g}B)-\mu(A)\mu(B)|<\epsilon\}$
is a central$^{*}$-set.
\end{enumerate}
\end{defn}

\begin{lem}
\label{st mix}Let $\left(X,\mathcal{B},\mu,(T_{f})_{f\in F_{q}[x]}\right)$
be a measure preserving system. Then it is strong mixing iff for each
$B\in{\cal B}$ with $\mu(B)>0$ and an infinite set $F$, there exists
a sequence of polynomials $\langle f_{n}\rangle_{n=1}^{\infty}$ in $F$
such that $\chi_{B}\circ T_{f_{n}}=U_{T_{f_{n}}}\chi_{B}\to f_{B}=\mu(B)$.
\end{lem}
\begin{proof}
Let $\left(X,\mathcal{B},\mu,(T_{f})_{f\in F_{q}[x]}\right)$ be a
measure preserving system. Let $\left\{ A_{i}\right\} _{i=1}^{\infty}$
be a countable basis of ${\cal B}$ i.e. $\left\{ A_{i}\right\} _{i=1}^{\infty}$
is dense in ${\cal B}$ with the metric $d(A,B)=\mu(A\bigtriangleup B)$.
Let $B\in{\cal B}$, with $\mu(B)>0$ and $F$ be an infinite set.
Let us set for each $i\in\mathbb{N}$ and $\epsilon>0$,
\[
F(i,\epsilon)=\{f\in F_{q}[x]:|\mu(A_{i}\cap T_{f}B)-\mu(A_{i})\mu(B)|<\epsilon\}.
\]

Then for each $i\in\mathbb{N}$, $F(i,\epsilon)$ is a cofinite set. Let
us choose one $f_{1}(x)\in F\cap F(1,1)$. Next we choose another
$f_{2}(x)\in F\cap F(1,\frac{1}{2})\cap F(2,\frac{1}{2})$
such that $\text{deg}f_{2}>\text{deg}f_{1}$. Inductively we choose
a sequence $\langle f_{n})_{n=1}\rangle ^{\infty}$ with $\text{deg}f_{n+1}>\text{deg}f_{n}$
such that

\[
f_{n+1}\in F\cap F(1,\frac{1}{n+1})\cap\ldots\cap F(i+1,\frac{1}{n+1}).
\]

So we get a subsequence $\langle f_{n}\rangle_{n=1}^{\infty}$ of $F$. By choosing
again a subsequence from it as our requirements we can assume $\chi_{B}\circ T_{f_{n}}=U_{T_{f_{n}}}\chi_{B}\to f_{B}\text{( weakly})$. It
is clear that for each $i$
\[
\int\chi_{A_{i}}(f_{B}-\mu(B))d\mu=0.
\]

This implies that $f_{B}=\mu(B)$.

Conversely given that, for each $B\in{\cal B}$ with $\mu(B)>0$ and an
infinite set $F$, there exists a sequence of polynomials $\langle f_{n}\rangle_{n=1}^{\infty}$
in $F$ such that $\chi_{B}\circ T_{f_{n}}=U_{T_{f_{n}}}\chi_{B}\to f_{B}=\mu(B)$.
Now if $\left(X,\mathcal{B},\mu,(T_{f})_{f\in F_{q}[x]}\right)$ is
not strong mixing then there exist $A,B\in{\cal B}$ with positive
measure and $\epsilon>0$ such that $\{f\in F_{q}[x]:|\mu(A\cap T_{f}B)-\mu(A)\mu(B)|\geq\epsilon\}$
is an infinite set. For this $B\in{\cal B}$, and the infinite subset
$F\subset F_{q}[x]$ we consider the set $\text{cl}_{w}\{U_{T_{f}}(\chi_{B}):f\in F\}$.
Then by the given hypothesis there is a constant function $f_{B}\in\text{cl}_{w}\{U_{T_{f}}(\chi_{B}):f\in F\}$.
Clearly the set $\{f\in F_{q}[x]:\mu(A\cap T_{f}B)\geq\mu(A)\mu(B)+\epsilon\}$
is infinite. Therefore each $f\in\text{cl}_{w}\{U_{T_{f}}(\chi_{B}):f\in F\}$
satisfies that $\int\chi_{A}\cdot fd\mu\geq\mu(A)\mu(B)+\epsilon$.
This contradicts the assumption  $\mu(B)\in\text{cl}_{w}\{U_{T_{f}}(\chi_{B}):f\in F\}$.
\end{proof}

\begin{thm}
Let $\left(X,\mathcal{B},\mu,(T_{f})_{f\in F_{q}[x]}\right)$ be a
measure preserving action. Then $T$ is strongly mixing iff for any
$\epsilon>0$ and $A\in{\cal B}$ with $\mu(A)>0$
\[
\{f\in F_{q}[x]:|\mu(A\cap T_{f}A)-\mu(A)^{2}|<\epsilon\}\in\triangle^{*}.
\]
\end{thm}
\begin{proof}
Strong mixing clearly implies the given condition.

For the converse, by Lemma \ref{st mix} it is sufficient to show that
for each $B\in{\cal B}$ with $\mu(B)>0$ and an infinite set $F$,
there exists a sequence of polynomials $\langle f_{n}\rangle_{n=1}^{\infty}$
in $F$ such that $\chi_{B}\circ T_{f_{n}}=U_{T_{f_{n}}}\chi_{B}{\longrightarrow}f_{B}=\mu(B)$ in the weak topology.
With out loss of generality we can assume that  $F$ has a sequence
$\langle f_{n}\rangle_{n=1}^{\infty}$ such that $\text{deg}f_{i}<\text{deg}f_{i+1}$.
 Since $\triangle$-sets
always have the Ramsey property, there exists some $F_{1}\subset F$ such
that
\[
F_{1}-F_{1}\subset(F-F)\cap\{f\in F_{q}[x]:|\mu(B\cap T_{f}B)-\mu(B)^{2}|<\frac{1}{2}\}.
\]

Choosing $F_{1}\supset F_{2}\supset\ldots\supset F_{k}$, we can inductively
choose $F_{k+1}\subset F_{k}$ such that

\[
F_{k+1}-F_{k+1}\subset(F_{k}-F_{k})\cap\{f\in F_{q}[x]:|\mu(B\cap T_{f}B)-\mu(B)^{2}|<\frac{1}{2^{k+1}}\}.
\]

Therefore for any $f\neq g\in F_{k}$, we have $|\mu(T_{f}B\cap T_{g}B)-\mu(B)^{2}|<\frac{1}{2^{k}}$.

Now let us consider a sequence $\langle f_{n_{i}}\rangle_{n=1}^{\infty}$
 such that $f_{n_{i}}\in F_{n_{i}}$. Then clearly $\chi_{B}\circ T_{f_{n_{i}}}=U_{T_{f_{n_{i}}}}{\longrightarrow}f_{B}$ in the weak topology.
Thus
\[
\langle f_{B},f_{B}\rangle=\lim_{i}\lim_{j}\langle\chi_{B}\circ T_{f_{n_{i}}},\chi_{B}\circ T_{f_{n_{j}}}\rangle\leq\mu(B)^{2}+\lim_{i}\frac{1}{2^{i}}=\mu(B)^{2}=\left(\int f_{B}d\mu\right)^{2}.
\]

This shows that $f_{B}=\mu(B)$ due to the Cauchy-Schwarz inequality.
\end{proof}

So the above theorem shows that like $(\mathbb{N},+)$ action, in case of
$(F_{q}[x],+)$ action also $\triangle^{*}$-mixing and strong mixing
are equivalent. But the authors believe that this is not true for the action of arbitrary group.

\section{Action of $(F_{q}[x],\cdot)$}

In this section we shall show that $(F_{q}[x],\cdot)$ behaves quite
similarly like $(\mathbb{N},+)$ and using some established examples for
the action of $(\mathbb{N},+)$, we shall produce some examples of $(F_{q}[x],\cdot)$.
First, we require the following lemmas.

\begin{lem}
Let $\varphi:(F_{q}[x],\cdot)\to(\mathbb{N},+)$ be a map defined by $\varphi(f)=\text{deg}f$.
Then $C$ is an IP-set in $(\mathbb{N},+)$ iff $\varphi^{-1}(C)$ is an IP-set
in $(F_{q}[x],\cdot)$.
\end{lem}

\begin{cor}
Let $\varphi:(F_{q}[x],\cdot)\to(\mathbb{N},+)$ be a map defined by $\varphi(f)=\text{deg}f$.
Then $C$ is an IP$^{*}$-set in $(\mathbb{N},+)$ iff $\varphi^{-1}(C)$ is an
IP$^{*}$-set in $(F_{q}[x],\cdot)$.
\end{cor}

The following lemma follows from \citep[Corollary 4.22]{refBG}. We still include a proof suggested by Prof. Neil Hindman, for the sake of completeness.

\begin{lem}
Let $\varphi:(F_{q}[x],\cdot)\to(\mathbb{N},+)$ be a map defined by $\varphi(f)=\text{deg}f$.
Then $C$ is a central set in $(\mathbb{N},+)$ iff $\varphi^{-1}(C)$ is a central set in
$(F_{q}[x],\cdot)$.
\end{lem}
\begin{proof}
Clearly $\varphi$ is a homomorphism so that it has continuous extension
$\tilde{\varphi}$ over $(\beta F_{q}[x],\cdot)$ onto $(\beta\mathbb{N},+)$ \citep[Corollary 4.22 and
Exercise 3.4.1]{refHS}.\\
Necessity. Let $C$ be a central set in $(\mathbb{N},+)$ and let $p$ be a minimal
idempotent containing $C$. Let $M=\tilde{\varphi}^{-1}(\{p\})$. Then $M$ is a compact Hausdorff right topological semigroup, so pick an idempotent $q \in K(M)$. We claim that $q$ is
minimal in $(\beta F_{q}[x],\cdot)$. So let $r$ be an idempotent in $(\beta F_{q}[x],\cdot)$ with $r \leq q$. Then $\varphi(r)\leq\varphi(q)=p$. Hence $r\in M$ and  thus $r=q$.\\
Sufficiency. Assume that $\varphi^{-1}(C)$ is central in $(\beta F_{q}[x],\cdot)$. Pick an idempotent $p \in K(\beta F_{q}[x])$
such that $\varphi^{-1}(C)\in p$. Then $\tilde{\varphi}(p)\in \tilde{\varphi}(K(\beta F_{q}[x]))$ and $\tilde{\varphi}(K(\beta F_{q}[x]))=K(\beta \mathbb{N})$ by \citep[Exercise 1.7.3]{refHS}.  Thus by \citep[Lemma 3.30]{refHS} $\varphi[\varphi^{-1}(C)]\in\tilde{\varphi}(p)$ and $\varphi[\varphi^{-1}(C)=C$.

\end{proof}

\begin{cor}
Let $\varphi:(F_{q}[x],\cdot)\to(\mathbb{N},+)$ be a map defined by $\varphi(f)=\text{deg}f$.
Then $C$ is a central$^{*}$ set in $(\mathbb{N},+)$ iff $\varphi^{-1}(C)$
is a central$^{*}$ set in $(F_{q}[x],\cdot)$.
\end{cor}

We know that strong mixing $\Rightarrow$ mild mixing $\Rightarrow$
weak mixing; but the converses are not true in
general. In case of the action of $(F_{q}[x],+)$ on a measure space
$(X,{\cal B},\mu)$ all the mixings are equivalent. In contrast to
the action of $(F_{q}[x],\cdot)$ we see that there are weak mixing systems
that are not mild mixing and there are mild mixing systems that are not
strong mixing.
Let $(X,{\cal B},\mu,T)$ be a weak mixing system which is not mild
mixing. We define an action of $(F_{q}[x],\cdot)$ by the formula
\[
T_{f}(x)=T^{\text{deg}f}(x).
\]
Let $\epsilon>0$ and any $A,B\in{\cal B}$ with positive measure.
Since $(X,{\cal B},\mu,T)$ is weak mixing
\[
\{n\in\mathbb{Z}:|\mu(A\cap T^{-n}B)-\mu(A)\mu(B)|<\epsilon\}
\]

is a central$^{*}$-set. This further implies that the set
\[
\{f:|\mu(A\cap T_{f}B)-\mu(A)\mu(B)|<\epsilon\}=\{f:|\mu(A\cap T^{\text{deg}f}B)-\mu(A)\mu(B)|<\epsilon\}
\]
is a central$^{*}$-set so that $(X,{\cal B},\mu,T_{f\in(F_{q}[x],\cdot)})$
is weak mixing. Using quite similar technique and the fact that a
set $A$ in $\mathbb{N}$ is an IP$^{*}$-set iff $\varphi^{-1}(A)$ is an
IP$^{*}$-set, we can show that $(X,{\cal B},\mu,T_{f\in(F_{q}[x],\cdot)})$
is not mild mixing. Similarly we can show that there exists $(X,{\cal B},\mu,T_{f\in(F_{q}[x],\cdot)})$
which is a mild mixing system but not strong mixing.
In \citep{refBH} the authors introduced the notion of  dynamical IP$^{*}$-set.
\begin{defn}
A set  $A$ in a semigroup $S$ is called a dynamical IP$^{*}$-set
if there exists a measure preserving dynamical system $\left(X,\mathcal{B},\mu,(T_{s})_{s\in S}\right)$,
$A\in{\cal B}$ with $\mu\left(A\right)>0$, such that $\left\{ s\in S:\mu\left(A\cap T_{s}^{-1}A\right)>0\right\} $$\subseteq A$.
\end{defn}

By \citep[Theorem 16.32]{refHS}, there is an IP$^{*}$ set $B$ in
$(\mathbb{N},+)$ such that for each $n\in\mathbb{N}$, neither $n+B$ nor $-n+B$
is an IP$^{*}$ set. Consequently, the following theorem shows that
not every IP$^{*}$-set is a dynamical IP$^{*}$-set.

\begin{thm}
Let B be a dynamical IP$^{*}$ set in $(\mathbb{N},+)$. There is a dynamical
IP$^{*}$-set $C\subset B$ such that for each $n\in C$, $-n+C$
is a dynamical IP$^{*}$ set (and hence not every IP$^{*}$ set is
a dynamical IP$^{*}$ set).
\end{thm}

\begin{proof}
\citep[Theorem 19.35][]{refHS}.
\end{proof}

Now, since there is an IP$^{*}$ set $B$ in $(\mathbb{N},+)$ such that for
each $n\in\mathbb{N}$, neither $n+B$ nor $-n+B$ is an IP$^{*}$ set, we
can show that $\varphi^{-1}B$ is an IP$^{*}$-set in $(F_{q}\left[x\right],\cdot)$,
but neither $\varphi^{-1}B/f$ nor $(\varphi^{-1}B)f$ is an IP$^{*}$-set
in $(F_{q}\left[x\right],\cdot)$ for any $f\in$ $F_{q}\left[x\right]\setminus F_{q}$.
Using an analogous method used in \citep[Theorem 19.35][]{refHS}, we
can prove the following.

\begin{thm}
Let B be a dynamical IP$^{*}$ set in $(F_{q}[x],\cdot)$. There
is a dynamical IP$^{*}$-set $C\subset B$ such that for each $f\in C$,
$C/f$ is a dynamical IP$^{*}$ set (and hence not every IP$^{*}$
set is a dynamical IP$^{*}$ set in $(F_{q}[x],\cdot)$ ).
\end{thm}

It can be proved that if $B$ be a dynamical IP$^{*}$ set  in $(\mathbb{N},+)$
then $\varphi^{-1}(C)$ is also a dynamical IP$^{*}$ set in $(F_{q}[x],\cdot)$
by the transformation $T_{f}(x)=T^{\text{deg}f}(x)$. But we do not
know whether the converse is true or false.
In the discussion of the action of $(F_{q}[x],\cdot)$, we have noticed
that it behaves quite similar to the action of $(\mathbb{N},+)$. This motivates
us to raise the following question.
\begin{qs}
Given any ergodic system $\left(X,\mathcal{B},\mu,T_{f\in(F_{q}[x],\cdot)}\right)$,
are the sets,

\[
\{f\in F_{q}[x]:\mu(A\cap T_{f}A\cap T_{f^{2}}A)>\mu(A)^{3}-\epsilon\}
\]

and
\[
\{f\in F_{q}\left[x\right]:\mu(A\cap T_{f}A\cap T_{f^{2}}A\cap T_{f^{3}}A>\mu(A)^{4}-\epsilon\}
\]
syndetic subsets of $(F_{q}[x],\cdot)$?
\end{qs}
In contrast, it can be shown that for $k>3$, $\{f\in F_{q}\left[x\right]:\mu(A\cap T_{f}A\cap T_{f^{2}}A\cap T_{f^{3}}A\cap\ldots\cap T_{f^{k}}A>\mu(A)^{k+1}-\epsilon\}$
are not syndetic subsets of $(F_{q}[x],\cdot)$ by the transformation
$T_{f}(x)=T^{\text{deg}f}(x)$.

%\section{Bibliography styles}
%
%There are various bibliography styles available. You can select the style of your choice in the preamble of this document. These styles are Elsevier styles based on standard styles like Harvard and Vancouver. Please use Bib\TeX\ to generate your bibliography and include DOIs whenever available.
%
%Here are two sample references: \cite{Feynman1963118,Dirac1953888}.
%
%\section*{References}
%
%
%\bibliography{mybibfile}

\end{document}